\documentclass[a4paper]{amsart}
\usepackage{tikz,pgfplots}
\usepackage[utf8]{inputenc} 
\usepackage{
	amsmath,amssymb, 
	graphicx, 
	mathtools,
	bm,
	todonotes,
	hypbmsec,
	float,
	subcaption,
	caption
}
\usepackage[shortlabels]{enumitem}

\makeatletter
\@namedef{subjclassname@2020}{\textup{2020} Mathematics Subject Classification}
\makeatother

\usepackage[colorlinks=true]{hyperref}
\usepackage{cleveref}

\usepackage[margin=3cm]{geometry}

\usepackage[doi=false,url=false,isbn=false]{biblatex}
\DeclarePairedDelimiterX\abs[1]\lvert\rvert{
	\ifblank{#1}{\:\cdot\:}{#1}
}
\DeclarePairedDelimiterX\norm[1]\lVert\rVert{ 
	\ifblank{#1}{\:\cdot\:}{#1}
}
\DeclarePairedDelimiterX{\inner}[2]{\langle}{\rangle}{ 
	\ifblank{#1}{\:\cdot\:}{#1},\ifblank{#2}{\:\cdot\:}{#2}
}

\providecommand\given{}
\DeclarePairedDelimiterX\Set[1]{\lbrace}{\rbrace}{
	\renewcommand\given{\SetSymbol[\delimsize]}
	#1
}
\DeclarePairedDelimiterXPP\Prob[1]{\mathbb{P}}(){}{
	\renewcommand\given{\nonscript\:\delimsize\vert\nonscript\:
		\mathopen{}}
	#1}
\DeclarePairedDelimiterXPP\Var[1]{\text{Var}}(){}{
	\renewcommand\given{\nonscript\:\delimsize\vert\nonscript\:
		\mathopen{}}
	#1}
\DeclarePairedDelimiterXPP\Mean[1]{\mathbb{E}}[]{}{
	\renewcommand\given{\nonscript\:\delimsize\vert\nonscript\:
		\mathopen{}}
	#1}

\newlist{exercise}{enumerate}{3}
\setlist[exercise]{wide, labelwidth=!, labelindent=0pt, label=\textbf{(\alph*)}}

\theoremstyle{plain}
\newtheorem{theorem}{Theorem}

\newtheorem{corollary}[theorem]{Corollary}
\newtheorem{proposition}[theorem]{Proposition}
\newtheorem{lemma}[theorem]{Lemma}
\newtheorem{assumption}{Assumption}

\newtheorem{remark}[theorem]{Remark}
\newtheorem{example}[theorem]{Example}

\newcommand\N{\mathbb{N}}
\newcommand\Z{\mathbb{Z}}

\newcommand\R{\mathbb{R}}

\usepackage[mathscr]{euscript} 

\DeclarePairedDelimiterX\lrangle[1]\langle\rangle{
	\ifblank{#1}{\:\cdot\:}{#1}
}

\newcommand\idd{\,\mathrm{d}} 
\newcommand\dd{\mathrm{d}} 
\newcommand\ff{\mathcal{F}}

\newcommand\B{\mathcal{B}} 
\newcommand\E{\mathcal{E}} 
 
\newcommand\D{\mathcal{D}}
\renewcommand\H{\mathcal{H}}
\newcommand\G{\mathcal{G}}

\newcommand{\p}{\mathbb P}
\newcommand{\e}{\mathbb E}

\newcommand{\stably}{\overset{st}{\rightarrow}}
\newcommand{\convd}{\overset{d}{\rightarrow}}
\newcommand{\cip}{\overset{\p}{\rightarrow}}

\newcommand{\ind}[1]{\mbox{\rm\large 1}_{{#1}}}

\begin{document}
\title{Local behavior of diffusions at the supremum}
\author[J.\ D.\ Thøstesen]{Jakob D.\ Thøstesen}
\begin{abstract}
This paper studies small-time behavior at the supremum of a diffusion process. For a solution to the SDE $\dd X_t=\mu(X_t)\dd t+\sigma(X_t)\dd W_t$ (where $W$ is a standard Brownian motion) we consider $(\epsilon^{-1/2}(X_{m^X+\epsilon t}-\overline{X}))_{t\in\R}$ as $\epsilon\downarrow0$, where $\overline{X}$ is the supremum of $X$ on the time interval $[0,1]$ and $m^X$ is the time of the supremum. It is shown that this process converges in law to a process $\hat{\xi}$, where $(\hat\xi_t)_{t\geq0}$ and $(\hat\xi_{-t})_{t\geq0}$ arise as independent Bessel-3 processes multiplied by $-\sigma(\overline{X})$. The proof is based on the fact that a continuous local martingale can be represented as a time-changed Brownian motion. This representation is also used to prove a limit theorem for zooming in on $X$ at a fixed time. As an application of the zooming-in result at the supremum we consider estimation of the supremum $\overline{X}$ based on observations at equidistant times.

\end{abstract}
\maketitle

\section{Introduction}

Differentiation is a central concept in classical analysis and it is useful in many areas with one example being approximation. When dealing with stochastic processes, however, we rarely care about differentiation as the paths of many typical processes are differentiable at few (if any) points. This means that there is a need for a similar tool to handle the local behavior of such processes.

A differentiation-type concept for stochastic processes was introduced in \cite{AGP95} with the purpose of describing local behavior at the supremum of the Brownian motion. This concept was revisited in \cite{iva_zooming} where it was called zooming in. A stochastic process $X$ starting at zero is said to satisfy the \emph{zooming-in condition} if
\begin{equation}\label{eq:zooming-in_assumption}
	(a_\epsilon X_{\epsilon t})_{t\geq0}\overset{fdd}{\rightarrow}(\tilde X_t)_{t\geq0}\qquad\text{as }\epsilon\downarrow0,
\end{equation}
where $a_\epsilon$ is a scaling function and $\tilde X$ is a non-trivial stochastic process. It is clear that this is connected to differentiation (from the right) at time 0. Indeed, if $t\mapsto X_t$ is differentiable from the right at $0$ then the convergence holds with $a_\epsilon=\epsilon^{-1}$ and $\tilde X$ being a line.

The related concept of zooming out was studied in \cite{lamperti62}. While this sounds like quite a different framework it is in fact possible to transfer many of ideas to the zooming-in setting. This includes the study of the scaling function and the limit process. For more details see \cite{iva_zooming}.

The zooming-in condition has proven to be a very useful regularity assumption in e.g.\ \cite{bis_iva,iva_pod,discretization}. In those papers the zooming-in theory plays a large role in various discretization problems.

Naturally there is a big difference between zooming in at a fixed time and at a random time. With $X$ being a Lévy process satisfying the zooming-in assumption it was shown in \cite{iva_zooming} that one may also zoom in at the supremum of $X$ over the interval $[0,1]$. The scaling is again $a_\epsilon$ and the law of the limit process is related to $\tilde X$. This theory was used in \cite{iva_pod} to derive limit theorems related to estimation of the supremum of $X$ in a high-frequency setting, and it was used in \cite{bis_iva} to study threshold exceedance for Lévy processes.

This paper presents limit results for zooming in at a fixed time and at the supremum of a diffusion process. Estimation of the supremum is studied as an application of the limit theory. The approach is based on the fact that a continuous local martingale can be represented as a time-changed Brownian motion. For zooming in at the supremum this lets us build on an existing zooming-in result for the Brownian motion.

All relevant definitions and prerequisites are contained in \S\ref{sec:def_and_prereq}. In \S\ref{sec:main_results} the main results are presented. Generality of the results and possible extensions are covered in \S\ref{sec:further_comments}, and finally the most technical proofs are found in \S\ref{sec:proofs}.

\section{Definitions and prerequisites}\label{sec:def_and_prereq}

\subsection{The setup}\label{sec:setup}

Consider the SDE
\begin{equation}\label{eq:SDEx}
	\dd X_t=\mu(X_t)\dd t+\sigma(X_t)\dd W_t\qquad\text{and}\qquad X_0=x_0,
\end{equation}
where $W$ is a standard Brownian motion. We assume that there exists a weak solution $(X,W)$ to \eqref{eq:SDEx}, defined on a filtered probability space $(\Omega,\ff,(\ff_t),\p)$ such that $X$ is $(\ff_t)$-adapted and $W$ is an $(\ff_t)$-Brownian motion. We assume that $(\ff_t)$ satisfies the usual conditions. In this paper we will encounter several $(\ff_t)$-adapted processes which are almost surely continuous, $X$ and $W$ being the first examples. Since $(\ff_t)$ is complete we may and will assume that these processes are continuous for all $\omega\in\Omega$.

We need some regularity assumptions on $\mu$ and $\sigma$ which are stated in Assumption~\ref{as:mu_sigma} below. Here, the range of $X$ is the set of points $x\in\R$ for which $\p(X_t=x\text{ for some }t\in[0,\infty))>0$. Note that the positivity in assumption (ii) is quite standard and guarantees the presence of some amount of noise at any time. This is important for zooming in since the presence of a Brownian motion affects the scaling function. For example, if $X$ is a Brownian motion plus a linear drift then $a_\epsilon\sim c_1\epsilon^{-1/2}$ (for some $c_1>0$), and if $X$ is just a linear drift then $a_\epsilon\sim c_2\epsilon^{-1}$ (for some $c_2>0$), see \cite[Thm.~2]{iva_zooming}.

\begin{assumption}\label{as:mu_sigma}
	\normalfont
	\begin{enumerate}[label=(\roman*)]
		\item[]
		\item The function $\mu\colon\R\to\R$ is locally bounded.
		\item The function $\sigma\colon\R\to[0,\infty)$ is continuous and strictly positive on the range of $X$.
	\end{enumerate}
\end{assumption}

We let $\overline{X}:=\sup_{t\in[0,1]}X_t$ denote the supremum of $X$ over the unit interval, and we denote the time of the ultimate supremum by $m^X:=\sup\Set{t\in[0,1]\given X_t=\overline{X}}$. We then define the pre- and post-supremum processes, $\underleftarrow{X}$ and $\underrightarrow{X}$, by
\begin{equation*}
	\underleftarrow{X}_t:=\begin{cases}
		X_{m^X-t}-\overline{X}&\text{if }0\leq t< m^X, \\
		\dagger&\text{if }t\geq m^X,
	\end{cases}
	\qquad\text{and}\qquad
	\underrightarrow{X}_t:=\begin{cases}
		X_{m^X+t}-\overline{X}&\text{if }0\leq t< 1-m^X, \\
		\dagger&\text{if }t\geq1-m^X.
	\end{cases}
\end{equation*}

\subsection{Path space and topology}

The processes appearing in this paper are viewed as random variables taking values in the measurable space $(D[0,\infty),\D)$, where $D[0,\infty)$ is the space of real-valued càdlàg functions defined on $[0,\infty)$ and $\D$ is the Borel $\sigma$-algebra induced by the Skorokhod topology. A standard reference treating this space is \cite[\S16]{billingsley}.

For convergence in distribution it is often sufficient to consider the restrictions of processes to intervals of the form $[0,T]$ for $T>0$. Consider $D[0,\infty)$-valued random variables (i.e.\ stochastic processes) $X,X^1,X^2,\dotsc$. Then $X^n\convd X$ if and only if $(X^n_t)_{t\in[0,T]}\convd(X_t)_{t\in[0,T]}$ for all $T>0$ where $X$ is almost surely continuous at $T$, see e.g.\ \cite[Thm.~16.7]{billingsley}. Here the restrictions are seen as random variables in $D[0,T]$ (the space of càdlàg functions on $[0,T]$).

\subsection{The central representation}\label{subsec:representation}

Suppose for a moment that $X$ solves the SDE \eqref{eq:SDEx} with $x_0=0$ and $\mu\equiv0$. Then $X$ is a continuous local $(\ff_t)$-martingale starting at zero. We denote the quadratic variation of $X$ by $[X]$ and recall that it is almost surely given by
\begin{equation*}
	[X]_t=\int_0^t \sigma^2(X_s)\idd s,\qquad t\geq0.
\end{equation*}
Note that $[X]$ is continuous and strictly increasing and denote its inverse by $\tau$. We define a new filtration $(\G_t)$ by $\G_t:=\ff_{\tau_t}$. A standard result (see e.g.\ \cite[Thm.~19.4]{kallenberg3}) gives the existence of a Brownian motion $\tilde W$ with respect to a standard extension $(\hat\G_t)$ of $(\G_t)$ (see \cite[p.~420]{kallenberg3}) such that $X=(\tilde W_{[X]_t})_{t\geq0}$ a.s. Furthermore, for any $s\geq0$ the random variable $[X]_s$ is a $(\G_t)_{t\geq0}$-stopping time.

\subsection{Stable convergence}

A central concept in this paper is the notion of \emph{stable convergence} which was originally introduced in \cite{Renyi63}. Later papers which are also of interest include \cite{aldous_eagleson,pod_vet}. In this subsection we present only the results which are relevant for this paper.

We consider a probability space $(\Omega,\ff,\p)$ supporting a sequence of random variables $(X_n)$ taking values in some Polish space. We say that $X_n$ converges stably to $X$ (written $X_n\stably X$) defined on an extension $(\tilde\Omega,\tilde\ff,\tilde\p)$ of the space if
\begin{equation}\label{eq:stable_def}
	\Mean{f(X_n)Z}\to\tilde\e[f(X)Z]
\end{equation}
for all bounded continuous functions $f$ and all bounded $\ff$-measurable $Z$.

The extension of $(\Omega,\ff,\p)$ is a product space $(\tilde\Omega,\tilde\ff)=(\Omega\times\Omega',\ff\otimes\ff')$ equipped with a probability measure $\tilde\p$ which satisfies $\tilde\p(A\times\Omega')=\p(A)$ for any $A\in\ff$. A random variable $Z$ defined on $(\Omega,\ff,\p)$ becomes a random variable on the extension by defining $Z(\omega,\omega'):=Z(\omega)$. We often need the extension to support a random variable $X$ which is independent of $\ff$. In that case we let $(\Omega',\ff',\p')$ be a probability space on which $X$ can be defined. As before $X$ can be viewed as a random variable on $(\Omega\times\Omega',\ff\otimes\ff')$, and taking $\tilde\p=\p\otimes\p'$ gives the desired independence. In this case, and when $X_n\stably X$, we sometimes say that the convergence is \emph{mixing}. This concept was first introduced in \cite{Renyi58}.

In order to work with stable convergence we need a few key results.

\begin{lemma}\label{lem:stable_properties}
	Assume that $X_n\stably X$. Then we have the following:
	\begin{enumerate}[label=(\roman*)]
		\item If $Y,Y_1,Y_2,\dotsc$ are random variables (taking values in some Polish space) on $(\Omega,\ff,\p)$ and $Y_n\cip Y$, then $(X_n,Y_n)\stably(X,Y)$.
		\item If $g$ is a Borel-measurable function taking values in a Polish space and $g$ is almost surely continuous at $X$ then $g(X_n)\stably g(X)$.
	\end{enumerate}
\end{lemma}
\begin{proof}
	See e.g.\ \cite[Thm.~3.18]{stablebook}.
\end{proof}

If $\H\subseteq\ff$ is a sub-$\sigma$-algebra and \eqref{eq:stable_def} is only known to hold for $\H$-measurable $Z$ we say that $X_n$ converges $\H$-stably to $X$ (written $X_n\overset{\H-st}{\rightarrow}X$). The following basic lemma shows that sometimes stable convergence can be obtained just by proving $\H$-stable convergence for a suitable sub-$\sigma$-algebra $\H$. This trick is used in e.g.\ the proof of \cite[Thm.~4.3.1]{jacod_protter}.

\begin{lemma}\label{lem:H-stable_to_F-stable}
	Let $\H\subseteq\ff$ be a sub-$\sigma$-algebra. Assume that each $X_n$ is $\H$-measurable, $X$ is independent of $\ff$ and $X_n\overset{\H-st}{\rightarrow}X$. Then $X_n\stably X$.
\end{lemma}
\begin{proof}
	We must verify \eqref{eq:stable_def} for all bounded continuous functions $f$ and all bounded $\ff$-measurable $Z$. Since $X_n$ is $\H$-measurable and $X_n\overset{\H-st}{\rightarrow}X$ it holds that
	\begin{equation*}
		\e[f(X_n)Z]=\e[f(X_n)\e[Z\mid\H]]\to\tilde\e[f(X)\e[Z\mid\H]].
	\end{equation*}
	Finally the assumed independence yields
	\begin{equation*}
		\tilde\e[f(X)\e[Z\mid\H]]=\tilde\e[f(X)]\tilde\e[Z]=\tilde\e[f(X)Z].
	\end{equation*}
\end{proof}

It is often useful to work with equivalent definitions of stable convergence.
\begin{lemma}\label{lem:stable_definitions}
	For a sub-$\sigma$-algebra $\H\subseteq\ff$ the following statements are equivalent:
	\begin{enumerate}[label=(\roman*)]
		\item $X_n\overset{\H-st}{\rightarrow}X$.
		\item $(X_n,Y)\overset{\H-st}{\rightarrow}(X,Y)$ for any $\H$-measurable $Y$ taking values in some Polish space.
		\item $(X_n,Y)\convd(X,Y)$ for any $\H$-measurable $Y$ taking values in some Polish space.
		\item $(X_n,\ind{F})\convd(X,\ind{F})$ for any $F\in\E$, where $\E\subseteq\H$ is closed under finite intersections and further satisfies $\Omega\in\E$ and $\sigma(\E)=\H$.
	\end{enumerate}
\end{lemma}
\begin{proof}
	For equivalence of (i)-(iii) see \cite[Prop.~1]{pod_vet}, and for equivalence of (i) and (iv) see \cite[Thm.~3.17]{stablebook}.
\end{proof}

Independence plays a large role for convergence of joint distributions. The following lemma shows that joint stable convergence can also be obtained under certain independence assumptions.

\begin{lemma}\label{lem:stable_conv_independence}
	Let $(X_n),(Y_n)$ be independent sequences of random variables, and let $X,Y$ be independent random variables such that $X$ and $Y$ are independent of $\ff$, $X_n\stably X$ and $Y_n\stably Y$. Then $(X_n,Y_n)\stably(X,Y)$.
\end{lemma}
\begin{proof}
	Let $\mathcal{A}=\sigma(\Set{X_n\given n\in\N})$, $\B=\sigma(\Set{Y_n\given n\in\N})$ and $\H=\sigma(\mathcal{A}\cup\B)$. According to Lemma~\ref{lem:H-stable_to_F-stable} it is sufficient to prove $\H$-stable convergence. For $A\in\mathcal{A}$ and $B\in\B$ we see that
	\begin{equation*}
		(X_n,\ind{A},Y_n,\ind{B})\convd(X,\ind{A},Y,\ind{B})
	\end{equation*}
	due to the assumed independence. Hence, $(X_n,Y_n,\ind{A\cap B})\convd(X,Y,\ind{A\cap B})$. The $\H$-stable convergence follows since condition (iv) in Lemma~\ref{lem:stable_definitions} is satisfied with $\E$ being the collection of sets on the form $A\cap B$ where $A\in\mathcal{A}$ and $B\in\B$.
\end{proof}

\section{Main results}\label{sec:main_results}

\subsection{Zooming in at a fixed time}

We begin with a limit theorem that formalizes the intuitive understanding of a diffusion process. Namely that the local behavior of $X$ at a fixed time $T>0$ is that of a scaled Brownian motion. To simplify we consider the time point $T=1$.

For $\epsilon>0$ and $t\in\R$ we let $X^{(\epsilon)}_t:=\epsilon^{-1/2}(X_{1+\epsilon t}-X_1)$. Consider further two standard Brownian motions $U^{(1)}$ and $U^{(2)}$ defined on an extension of $(\Omega,\ff,\p)$ which are independent of each other and of $\ff$.

\begin{theorem}\label{thm:zoom_at_1}
	It holds that
	\begin{equation*}
		\big((X^{(\epsilon)}_{-t})_{t\geq0},(X^{(\epsilon)}_t)_{t\geq0}\big)\stably\big(\sigma(X_1)U^{(1)},\sigma(X_1)U^{(2)}\big)\qquad\text{as }\epsilon\downarrow0.
	\end{equation*}
\end{theorem}

Dealing with $(X^{(\epsilon)}_t)_{t\geq0}$ is fairly simple as we look forward in time. Looking backwards in time is generally harder and proving the convergence of $(X^{(\epsilon)}_{-t})_{t\geq0}$ is indeed rather technical. The proof of Theorem~\ref{thm:zoom_at_1} is deferred to \S\ref{sec:proof_zoom_at_1}.

Looking backwards in time may be difficult but it is quite useful. The following result is very intuitive in addition to being necessary for proving Theorem~\ref{thm:zoom_sup_diffusion} below, and proving it is now trivial.

\begin{corollary}\label{cor:max_not_at_1}
	Almost surely $m^X\neq1$.
\end{corollary}
\begin{proof}
	Let $A\subseteq D[0,\infty)$ be the set of functions $f$ in $D[0,\infty)$ with $f(t)\leq0$ for all $t\in[0,1)$. Using \cite[Thm.~16.1]{billingsley} it is easy to verify that $A$ is closed in the Skorokhod topology. It follows from Theorem~\ref{thm:zoom_at_1} and the Portmanteau theorem that
	\begin{equation*}
		\Prob{m^X=1}\leq\limsup_{\epsilon\downarrow0}\Prob{(X^{(\epsilon)}_{-t})_{t\geq0}\in A}\leq\tilde\p((\sigma(X_1)U^{(1)}_t)_{t\geq0}\in A)=0.
	\end{equation*}
\end{proof}

\subsection{Zooming in at the supremum}

The local behavior of $X$ at time $1$ is described by the zooming-in result in Theorem~\ref{thm:zoom_at_1}. In a similar fashion we want to describe the local behavior at the supremum through a zooming-in result. It is well-known (see e.g.\ \cite{bertoin93}) that the negated pre- and post-supremum processes for a Brownian motion are two independent Bessel-3 processes (killed at certain random times). With this in mind the following result is somewhat intuitive.

\begin{theorem}\label{thm:zoom_sup_diffusion}
	Let $B^{(1)}$ and $B^{(2)}$ be two independent Bessel-3 processes defined on an extension of $(\Omega,\ff,\p)$ such that both processes are independent of $\ff$. Then it holds that
	\begin{equation}\label{eq:zoom_sup_diffusion}
		\big((\epsilon^{-1/2}\underleftarrow{X}_{\epsilon t})_{t\geq0},(\epsilon^{-1/2}\underrightarrow{X}_{\epsilon t})_{t\geq0}\big)\stably\big(-\sigma(\overline{X})B^{(1)},-\sigma(\overline{X})B^{(2)}\big)\qquad\text{as }\epsilon\downarrow0.
	\end{equation}
\end{theorem}

The proof of Theorem~\ref{thm:zoom_sup_diffusion} is deferred to \S\ref{sec:proof_zoom_sup_diffusion}.

\subsection{Estimation of the supremum}

As an application of Theorem~\ref{thm:zoom_sup_diffusion} we consider a high-frequency setting in which the process $X$ is observed on the set of times $\epsilon(\N_0+U)\cap[0,1]$ for some small $\epsilon>0$, where $U$ is a standard uniform defined on an extension of the space such that it is independent of $\ff$ and $B^{(1)},B^{(2)}$. The objective is to estimate the supremum $\overline{X}$ over $[0,1]$. To avoid constantly having to intersect with the unit interval we consider $X$ as being restricted to this interval.

We take the basic estimator $M^{(\epsilon)}:=\sup_{t\in\epsilon(\N_0+U)}X_t$. The following result establishes the convergence rate $\epsilon^{-1/2}$.
\begin{proposition}\label{prop:discretization_bounds_U}
	For all $\epsilon>0$ it holds that
	\begin{equation*}
		0\geq\epsilon^{-1/2}(M^{(\epsilon)}-\overline{X})\geq\epsilon^{-1/2}\underrightarrow{X}_{\epsilon\{U-m^X/\epsilon\}},
	\end{equation*}
	where $\{U-m^X/\epsilon\}$ is the fractional part of $U-m^X/\epsilon$.

	Furthermore, there is stable convergence of the lower bound:
	\begin{equation*}
		\epsilon^{-1/2}\underrightarrow{X}_{\epsilon\{U-m^X/\epsilon\}}\stably-\sigma(\overline{X})B^{(2)}_{U}.
	\end{equation*}
\end{proposition}
\begin{proof}
	Observe that
	\begin{equation*}
		\epsilon^{-1/2}(M^{(\epsilon)}-\overline{X})=\sup_{i\in\N_0}\epsilon^{-1/2}(X_{\epsilon(i+U)}-\overline{X})=\sup_{i\in\Z}\epsilon^{-1/2}(X_{\epsilon(i+\{U-m^X/\epsilon\})+m^X}-\overline{X})
	\end{equation*}
	for all $\epsilon>0$. We can get a lower bound by taking a specific $i$ instead of taking the supremum over $\Z$. With $i=0$ we get the claimed lower bound.

	By conditioning one sees that for all $\epsilon>0$ the fractional part $U_\epsilon:=\{U-m^X/\epsilon\}$ is a standard uniform independent of $\ff$ and $B^{(1)},B^{(2)}$. In combination with Theorem~\ref{thm:zoom_sup_diffusion} and \cite[Prop.~13.2.1]{whitt} we obtain the convergence of the lower bound.
\end{proof}

\begin{remark}\label{rem:estimation_lower_bound_U}
	\normalfont

	The lower bound in Proposition~\ref{prop:discretization_bounds_U} is somewhat conservative. Indeed, in the proof we see that the discretization error can be written as $\sup_{i\in\Z}\epsilon^{-1/2}(X_{\epsilon(i+\{U-m^X/\epsilon\})+m^X}-\overline{X})$. Looking to Theorem~\ref{thm:zoom_sup_diffusion} it is expected that this quantity will converge to $\sup_{i\in\Z}\hat\xi_{i+U}$, where $\hat\xi_t=-\sigma(\overline{X})B^{(1)}_{-t}$ for $t<0$ and $\hat\xi_t=-\sigma(\overline{X})B^{(2)}_t$ for $t\geq0$. However, this is not straight-forward to prove. The issue is that taking the supremum over an unbounded set of times is not continuous. This was solved in \cite[App.~B]{bis_iva} where the authors corrected the proof of \cite[Thm.~5]{iva_zooming}. In those papers $X$ is a Lévy process satisfying the zooming-in condition. The approach is not directly applicable here because it is based on results known only for Lévy processes.
\end{remark}

It is perfectly valid to ask why we choose to sample at times $\epsilon(i+U)$ rather than $\epsilon i$ for $i\in\N_0$. In the latter case one would consider the estimator $\tilde M^{(\epsilon)}:=\sup_{t\in\epsilon\N_0}X_t$. For this estimator it holds that
\begin{equation*}
	\epsilon^{-1/2}(\tilde M^{(\epsilon)}-\overline{X})=\sup_{i\in\Z}\epsilon^{-1/2}(X_{\epsilon(i+\{-m^X/\epsilon\})+m^X}-\overline{X})
\end{equation*}
for any $\epsilon>0$. This gives the lower bound $\epsilon^{-1/2}\underrightarrow{X}_{\epsilon\{-m^X/\epsilon\}}$. In order to obtain a limit theorem for this quantity we need to know what happens to $\{-m^X/\epsilon\}$ as $\epsilon\downarrow0$. By the classical result of \cite{K37} it is known that $\{-m^X/\epsilon\}$ converges to the standard uniform distribution if $m^X$ has a density wrt.\ the Lebesgue measure. As seen in Proposition~\ref{prop:discretization_bounds_U} we are able to avoid such considerations by translating the sampling times by $\epsilon U$.

\section{Further comments}\label{sec:further_comments}

\subsection{Generality of the results}

Theorem~\ref{thm:zoom_at_1} describes zooming in at time $1$. Naturally there is nothing special about the time $1$ so the result also holds if we zoom in at some other fixed time $T>0$. In that case one simply replaces $\sigma(X_1)$ by $\sigma(X_T)$ in the limit. The time point $1$ is chosen only to simplify notation. 

In the same way there is nothing special about the time interval $[0,1]$ in the formulation of Theorem~\ref{thm:zoom_sup_diffusion}. This interval can be replaced by $[T_1,T_2]$ where $0\leq T_1<T_2<\infty$ are fixed. In the formulation of the result one will then have to define $\overline{X}:=\sup_{t\in[T_1,T_2]}X_t$.

\subsection{Extending to other classes of stochastic processes}

The approach used to prove Theorem~\ref{thm:zoom_at_1} and Theorem~\ref{thm:zoom_sup_diffusion} is based on representing the local martingale part of $X$ as a time-changed Brownian motion. The time-change is differentiable and this lets us apply zooming-in results for the Brownian motion to obtain corresponding results for $X$.

It is possible to extend the result about zooming in at the supremum to other classes of stochastic processes. In \cite{iva_zooming} this was done for any Lévy process satisfying the zooming-in condition \eqref{eq:zooming-in_assumption}. With the approach used to prove Theorem~\ref{thm:zoom_sup_diffusion} it is likely that this result can be used to prove limit results for zooming in at the supremum of time-changed Lévy processes. Below are two examples where this appears to be do-able.
\begin{example}
\normalfont
\begin{enumerate}[label=\bf({\Alph*}),wide,labelwidth=!,labelindent=0pt]
\item[]
\item Let $X$ be a positive $1/\alpha$-self-similar Markov process (pssMp) starting at some value $x>0$. The classical result of \cite{lamperti72} tells us that there exists a Lévy process $\xi$ such that
\begin{equation*}
	X_t=x\exp(\xi_{\tau(tx^{-\alpha})}),\qquad t\geq0,
\end{equation*}
where $\tau(tx^{-\alpha})=\inf\Set{s>0\given \int_0^s\exp(\alpha\xi_u)\idd u\geq tx^{-\alpha}}$. The key point is that $X_t$ is obtained by time-changing a Lévy process and applying a strictly increasing and differentiable function. Note also that the time-change is differentiable. The last ingredient is that $\xi$ must satisfy the zooming-in condition. This is completely characterized in \cite[Thm.~2]{iva_zooming} in terms of the characteristics of $\xi$. Note also that one must pay special attention to a possible jump at the time of supremum.

\item[]

\item Let $X$ be a continuous-state branching process. Then there exists (see e.g.\ \cite[Thm.~10.2]{kyprianou06}) a Lévy process $\zeta$ such that
\begin{equation*}
	X_t=\zeta_{\theta(t)\wedge\tau_0^-},\qquad t\geq0,
\end{equation*}
where $\tau_0^-=\inf\Set{s>0\given \zeta_s<0}$ and $\theta(t)=\inf\Set{s>0\given\int_0^s\zeta_u^{-1}\idd u>t}$. We note that the time-change is not as well-behaved as for the class of pssMps. For example, $t\mapsto\theta(t)$ is not differentiable everywhere. As a consequence one will again have to be particularly aware of any jump at the supremum.
\end{enumerate}
\end{example}

\section{Proofs}\label{sec:proofs}

\subsection{Proof of Theorem~\ref{thm:zoom_at_1}}\label{sec:proof_zoom_at_1}

As in the formulation of Theorem \ref{thm:zoom_at_1} we let $(U^{(1)},U^{(2)})$ denote a pair of independent standard Brownian motions, defined on an extension of $(\Omega,\ff,\p)$ such that they are also independent of $\ff$.

We may write $X_t$ as
\begin{equation*}
	X_t=x_0+A_t+M_t,\qquad t\geq0,
\end{equation*}
where $A$ is a continuous and $(\ff_t)$-adapted process with bounded variation, $M$ is a continuous $(\ff_t)$-local martingale and $A_0=M_0=0$ a.s. We see that
\begin{equation*}
	X^{(\epsilon)}_t=\epsilon^{-1/2}(X_{1+\epsilon t}-X_1)=\epsilon^{-1/2}(A_{1+\epsilon t}-A_1)+\epsilon^{-1/2}(M_{1+\epsilon t}-M_1)
\end{equation*}
for all $t\geq-1/\epsilon$. We treat each term from the right-hand side separately.

Note that $A_t=\int_0^t\mu(X_s)\idd s$ for all $t\geq0$ a.s. Since $\mu$ and $X$ are both locally bounded we immediately find that
\begin{equation*}
	\sup_{t\in[-T,T]}\epsilon^{-1/2}\abs{A_{1+\epsilon t}-A_1}\leq2T\epsilon^{1/2}\sup_{t\in[1-\epsilon T,1+\epsilon T]}\abs{\mu(X_t)}\to0
\end{equation*}
a.s.\ as $\epsilon\downarrow0$ for any $T>0$.

Below in the proof of Theorem~\ref{thm:zoom_sup_diffusion} it is necessary to deal with the drift differently. The same approach could be used here, however it is the author's belief that the calculation above is more illustrative since it clearly shows that the drift vanishes due to the $\epsilon^{-1/2}$ scaling.

It remains to show that
\begin{equation}\label{eq:zoom_M_at_1}
	\big((\epsilon^{-1/2}(M_{1-\epsilon t}-M_1))_{t\geq0},(\epsilon^{-1/2}(M_{1+\epsilon t}-M_1))_{t\geq0}\big)\stably\big(\sigma(X_1)U^{(1)},\sigma(X_1)U^{(2)}\big).
\end{equation}
To do so we will represent $M$ as a time-changed Brownian motion. Let $(\ff^M_t)$ denote the completed natural filtration generated by $M$, let $\tau$ denote the inverse of $[M]$, and define $\G^M_t=\ff^M_{\tau_t}$. Now, as in \S\ref{subsec:representation} a standard result gives the existence of a Brownian motion $\tilde W$ with respect to a standard extension $(\hat\G^M_t)$ of $(\G^M_t)$ such that $M=(\tilde W_{[M]_t})_{t\geq0}$ a.s. Recall that $[M]_s$ is a $(\G^M_t)$-stopping time for any $s\geq0$. Finally we note that the quadratic variation of $M$ is given by
\begin{equation*}
	[M]_t=[X]_t=\int_0^t \sigma^2(X_s)\idd s,\qquad t\geq0
\end{equation*}
almost surely.

The next step in the proof of Theorem~\ref{thm:zoom_at_1} is Lemma~\ref{lem:zoom_in_from_right} below which allows for zooming in on $\tilde W$ from the right. Instead of simply zooming in at time $1$ we generalize to zooming in at $1-\epsilon R$ with $R\geq0$ since we will need this in the proof of Lemma~\ref{lem:zoom_BM_stopping_time} below. This slight generalization requires very little extra effort.

\begin{lemma}\label{lem:zoom_in_from_right}
	For any $R\geq0$ it holds that
	\begin{equation}\label{eq:zoom_in_from_right_0}
		(\epsilon^{-1/2}(\tilde W_{[M]_{1-\epsilon R}+\epsilon t}-\tilde W_{[M]_{1-\epsilon R}}))_{t\geq0}\stably U,
	\end{equation}
	where $U$ is a standard Brownian motion defined on an extension of $(\Omega,\ff,\p)$ such that $U$ is independent of $\ff$.
\end{lemma}
\begin{proof}
	We fix $R\geq0$ and recall that $[M]_{1-\epsilon R}$ is a $(\G^M_t)$-stopping time. It follows that the left-hand side of \eqref{eq:zoom_in_from_right_0} is a standard Brownian motion for any $\epsilon>0$ so the convergence in distribution is trivial.

	Now, let $\mathcal{A}$ denote the $\sigma$-algebra generated by the process $\tilde W':=(\tilde W_{[M]_1+t}-\tilde W_{[M]_1})_{t\geq0}$. The first step is proving $\mathcal{A}$-stable convergence. It is sufficient to show that
	\begin{equation}\label{eq:zoom_in_from_right_1}
		\big((\epsilon^{-1/2}(\tilde W_{[M]_{1-\epsilon R}+\epsilon t}-\tilde W_{[M]_{1-\epsilon R}}))_{t\in[0,T]},(\tilde W'_{t_i})_{i=1,\dotsc,k}\big)\convd \big((U_t)_{t\in[0,T]},(\tilde W'_{t_i})_{i=1,\dotsc,k}\big)
	\end{equation}
	for any $T>0$, $k\in\N$ and $0<t_1<\dotsc<t_k$. To this end define $a_i^{(\epsilon)}:=\tilde W_{[M]_1+t_i}-\tilde W_{[M]_1+\epsilon T}$ and $b_i^{(\epsilon)}:=\tilde W_{[M]_1+\epsilon T}-\tilde W_{[M]_1}$. Then $\tilde W'_{t_i}=a_i^{(\epsilon)}+b_i^{(\epsilon)}$, $b_i^{(\epsilon)}\to0$ a.s.\ as $\epsilon\downarrow0$, and for $\epsilon\in(0,t_1/T)$ we see that $a_i^{(\epsilon)}$ is independent of $(\tilde W_t)_{t\in[0,[M]_{1-\epsilon R}+\epsilon T]}$. Hence,
	\begin{equation*}
		\big((\epsilon^{-1/2}(\tilde W_{[M]_{1-\epsilon R}+\epsilon t}-\tilde W_{[M]_{1-\epsilon R}}))_{t\in[0,T]},(a_i^{(\epsilon)})_{i=1,\dotsc,k}\big)\convd \big((U_t)_{t\in[0,T]},(\tilde W'_{t_i})_{i=1,\dotsc,k}\big).
	\end{equation*}
	The convergence in \eqref{eq:zoom_in_from_right_1} follows immediately. This establishes \eqref{eq:zoom_in_from_right_0} with $\stably$ replaced by $\overset{\mathcal{A}-st}{\rightarrow}$.

	We let $\H:=\sigma(\G^M_{[M]_1}\cup\mathcal{A})=\sigma(\ff^M_1\cup\mathcal{A})$ and note that the left-hand side in \eqref{eq:zoom_in_from_right_0} is $\H$-measurable. Thus, proving $\H$-stable convergence automatically yields $\ff$-stable convergence by Lemma~\ref{lem:H-stable_to_F-stable}. We note that $\ff^M_1=\sigma(\bigcup_{\delta>0}\ff^M_{1-\delta})$ since $M(\omega)$ is continuous for all $\omega\in\Omega$ (recall the considerations in the beginning of \S\ref{sec:setup}). According to Lemma~\ref{lem:stable_definitions} it is sufficient to show that
	\begin{equation*}
		\big((\epsilon^{-1/2}(\tilde W_{[M]_{1-\epsilon R}+\epsilon t}-\tilde W_{[M]_{1-\epsilon R}}))_{t\geq0},\ind{A},\ind{F}\big)\convd(U,\ind{A},\ind{F})
	\end{equation*}
	for any $\delta>0$, $F\in\ff^M_{1-\delta}$ and $A\in\mathcal{A}$. Since the first two components on the left-hand side are independent of $\ind{F}$ for small enough $\epsilon$ this is a trivial consequence of the $\mathcal{A}$-stable convergence. This concludes the proof.
\end{proof}

We proceed by proving the following lemma, stating that we can zoom in on $\tilde W$ at time $[M]_1$. The proof follows the same strategy as the proof of \cite[Thm.~3]{discretization}.

\begin{lemma}\label{lem:zoom_BM_stopping_time}
	As $\epsilon\downarrow0$ it holds that
	\begin{equation}\label{eq:zoom_BM_stopping_time_0}
		\big((\tilde W^{(\epsilon)}_{-t})_{t\geq0},(\tilde W^{(\epsilon)}_t)_{t\geq0}\big)\stably(U^{(1)},U^{(2)}),
	\end{equation}
	where $\tilde W^{(\epsilon)}_t:=\epsilon^{-1/2}(\tilde W_{[M]_1+\epsilon t}-\tilde W_{[M]_1})$.
\end{lemma}
\begin{proof}
	There are two immediate things to note. Firstly, the convergence $(\tilde W^{(\epsilon)}_t)_{t\geq0}\stably U^{(2)}$ is nothing more than the case $R=0$ in Lemma~\ref{lem:zoom_in_from_right}. Secondly, since $(\tilde W^{(\epsilon)}_{-t})_{t\geq0}$ and $(\tilde W^{(\epsilon)}_t)_{t\geq0}$ are independent for all $\epsilon>0$ it is sufficient, according to Lemma~\ref{lem:stable_conv_independence}, to show that the former converges stably to $U^{(1)}$. Again it is sufficient to show stable convergence of the process restricted to the time interval $[0,T]$ for all $T>0$.

	For any $R\geq0$ we have the almost sure convergence $\epsilon^{-1}([M]_1-[M]_{1-\epsilon R})\to R\sigma^2(X_1)=:s$. Given $T>0$ we can pick $R$ such that $s>T$ with probability arbitrarily close to 1. With $Y^{(\epsilon)}_t=\epsilon^{-1/2}(\tilde W_{[M]_{1-\epsilon R}+\epsilon t}-\tilde W_{[M]_{1-\epsilon R}})$ we then write
	\begin{equation}\label{eq:zoom_BM_stopping_time_1}
		\epsilon^{-1/2}(\tilde W_{[M]_1-\epsilon t}-\tilde W_{[M]_1})=-(Y^{(\epsilon)}_{\epsilon^{-1}([M]_1-[M]_{1-\epsilon R})}-Y^{(\epsilon)}_{\epsilon^{-1}([M]_1-[M]_{1-\epsilon R}-\epsilon t)}).
	\end{equation}
	That is, on $\Set{s>T}$ the increment of $\epsilon^{-1/2}\tilde W$ over $[[M]_1-\epsilon t,[M]_1]$ can be viewed as the increment of $Y^{(\epsilon)}$ over $[\epsilon^{-1}([M]_1-[M]_{1-\epsilon R}-\epsilon t),\epsilon^{-1}([M]_1-[M]_{1-\epsilon R})]$ (for small enough $\epsilon>0$).

	Almost surely $\epsilon^{-1}([M]_1-[M]_{1-\epsilon R}-\epsilon t)\to s-t$ uniformly for $t\in[0,T]$. By combining this with \eqref{eq:zoom_BM_stopping_time_1}, Lemma \ref{lem:zoom_in_from_right}, continuity of subordination (see \cite[Thm.~13.2.2]{whitt}) and Lemma \ref{lem:stable_properties} we find that
	\begin{equation}\label{eq:zoom_BM_stopping_time_2}
		\Mean{\ind{\Set{s>T}}f((\tilde W^{(\epsilon)}_{-t})_{t\in[0,T]})Z}\to\tilde\e[\ind{\Set{s>T}}f(-(U_s-U_{s-t})_{t\in[0,T]})Z]
	\end{equation}
	for all bounded continuous $f$ and all bounded $\ff$-measurable $Z$, where $U$ is a standard Brownian motion defined on an extension of $(\Omega,\ff,\p)$ such that $U$ is independent of $\ff$ and independent of $U^{(2)}$. We conclude by noting that the limit in \eqref{eq:zoom_BM_stopping_time_2} is equal to $\tilde\e[\ind{\Set{s>T}}f((U^{(1)}_t)_{t\in[0,T]})Z]$, where $U^{(1)}$ is a standard Brownian motion defined on an extension of $(\Omega,\ff,\p)$, independent of $\ff$ and independent of $U^{(2)}$.
\end{proof}

We are now ready to finish the proof of Theorem~\ref{thm:zoom_at_1} which we have reduced to proving the convergence
\begin{equation*}
	\big((M^{(\epsilon)}_{-t})_{t\geq0},(M^{(\epsilon)}_t)_{t\geq0}\big)\stably\big(\sigma(X_1)U^{(1)},\sigma(X_1)U^{(2)}\big),
\end{equation*}
where $M^{(\epsilon)}_t:=\epsilon^{-1/2}(M_{1+\epsilon t}-M_1)$.

Firstly, we have the almost sure convergence
\begin{equation*}
	\sigma_\epsilon^2(t):=\epsilon^{-1}([M]_{1+\epsilon t}-[M]_1)\to t\sigma^2(X_1).
\end{equation*}
This convergence is uniform in $t$ over compact intervals so we get the a.s.\ functional convergence
\begin{equation*}
	\big((\sigma_\epsilon^2(-t))_{t\geq0},(\sigma_\epsilon^2(t))_{t\geq0}\big)\to\big((-t\sigma^2(X_1))_{t\geq0},(t\sigma^2(X_1))_{t\geq0}\big),
\end{equation*}
which we may add to the stable convergence in \eqref{eq:zoom_BM_stopping_time_0}.

Now, for $t\in\R$ we can write
\begin{equation*}
	M^{(\epsilon)}_t=\epsilon^{-1/2}(M_{1+\epsilon t}-M_1)=\epsilon^{-1/2}(\tilde W_{[M]_{1+\epsilon t}}-\tilde W_{[M]_1})=\tilde W^{(\epsilon)}_{\sigma_\epsilon^2(t)},
\end{equation*}
where $\tilde W^{(\epsilon)}$ is defined in Lemma~\ref{lem:zoom_BM_stopping_time}. 
By piecing the above together we obtain the convergence
\begin{equation*}
	\big((M^{(\epsilon)}_{-t})_{t\geq0},(M^{(\epsilon)}_t)_{t\geq0}\big)\stably\big((U^{(1)}_{t\sigma^2(X_1)})_{t\geq0},(U^{(2)}_{t\sigma^2(X_1)})_{t\geq0}\big)=\big(\sigma(X_1)\tilde U^{(1)},\sigma(X_1)\tilde U^{(2)}\big),
\end{equation*}
where $\tilde U^{(i)}_t:=\sigma^{-1}(X_1)U^{(i)}_{t\sigma^2(X_1)}$. Again we use continuity of subordination (see \cite[Thm.~13.2.2]{whitt}). We conclude by remarking that $(\tilde U^{(1)},\tilde U^{(2)})$ is again a pair of independent standard Brownian motions, also independent of $\ff$.

\subsection{Proof of Theorem~\ref{thm:zoom_sup_diffusion}}\label{sec:proof_zoom_sup_diffusion}

We begin by establishing that we may assume that $X$ starts at zero and has no drift. As in Theorem \ref{thm:zoom_sup_diffusion} $(B^{(1)},B^{(2)})$ denotes a pair of independent Bessel-3 processes, defined on an extension of $(\Omega,\ff,\p)$ such that they are also independent of $\ff$.

Following \cite[Ch.~33]{kallenberg3} we let $p$ be the function given by
\begin{equation*}
	p'(x)=\exp\left\{-2\int_{x_0}^x(\mu/\sigma^2)(u)\idd u\right\}\qquad\text{and}\qquad p(x_0)=0.
\end{equation*}
Note that this definition of $p$ has a problem at a value $x$ if the function $\mu/\sigma^2$ is not integrable over the interval $[x_0,x]$ (or $[x,x_0]$ depending on which is larger). However, if $x$ is in the range of $X$ then $\mu/\sigma^2$ is bounded on $[x_0,x]$ (or $[x,x_0]$) due to Assumption \eqref{as:mu_sigma}. As we will only need to evaluate $p$ at such points we need not worry.

Now, let $Y_t:=p(X_t)$ for $t\geq0$. The choice of $p$ has two particularly useful implications. Firstly, $p$ is strictly increasing so $\overline{Y}=p(\overline{X})$ and $m^X=m^Y$. Secondly, $Y$ is a diffusion process solving the SDE
\begin{equation}\label{eq:SDE_Y}
	\dd Y_t=\tilde\sigma(Y_t)\dd W_t\qquad\text{and}\qquad Y_0=0,
\end{equation}
where $\tilde\sigma=(\sigma p')\circ p^{-1}$.

Now we are able to prove the following lemma which is an essential step in proving Theorem~\ref{thm:zoom_sup_diffusion}.

\begin{lemma}\label{lem:nodrift}
	It is sufficient to prove Theorem \ref{thm:zoom_sup_diffusion} under the assumption that $x_0=0$ and $\mu\equiv0$.
\end{lemma}
\begin{proof}
	Assume that Theorem \ref{thm:zoom_sup_diffusion} holds for any diffusion process which starts at zero, has no drift and satisfies Assumption \ref{as:mu_sigma}.

	We consider the transformation $Y:=p(X)$ introduced above. In addition to solving the SDE \eqref{eq:SDE_Y} we further note that $Y$ satisfies Assumption~\ref{as:mu_sigma}. So by our initial assumption there is the convergence
	\begin{equation*}
		\left((\epsilon^{-1/2}\underleftarrow{Y}_{\epsilon t})_{t\geq0},(\epsilon^{-1/2}\underrightarrow{Y}_{\epsilon t})_{t\geq0}\right)\stably(-\tilde\sigma(\overline{Y})B^{(1)},-\tilde\sigma(\overline{Y})B^{(2)}),
	\end{equation*}
	where $\underleftarrow{Y}$ and $\underrightarrow{Y}$ are pre- and post-supremum processes defined for the interval $[0,1]$. Using the mean value theorem we find that
	\begin{equation*}
		\epsilon^{-1/2}\underrightarrow{X}_{\epsilon t}=\epsilon^{-1/2}(p^{-1})'(c_\epsilon(t))\underrightarrow{Y}_{\epsilon t},
	\end{equation*}
	where $c_\epsilon(t)$ is between $Y_{m^X+\epsilon t}$ and $\overline{Y}$. One easily verifies that $(p^{-1})'(c_\epsilon(\cdot))$ converges (in the Skorokhod topology) to the constant function $(p^{-1})'(\overline{Y})$. Hence,
	\begin{equation*}
		(\epsilon^{-1/2}(p^{-1})'(c_\epsilon(t))\underrightarrow{Y}_{\epsilon t})_{t\geq0}\stably-(p^{-1})'(\overline{Y})\tilde\sigma(\overline{Y})B^{(2)}=-\sigma(\overline{X})B^{(2)},
	\end{equation*}
	where the final identity comes from the definition of $\tilde\sigma$. Obviously we can do similar calculations for the pre-supremum process. Hence,
	\begin{equation*}
		\left((\epsilon^{-1/2}\underleftarrow{X}_{\epsilon t})_{t\geq0},(\epsilon^{-1/2}\underrightarrow{X}_{\epsilon t})_{t\geq0}\right)\stably(-\sigma(\overline{X})B^{(1)},-\sigma(\overline{X})B^{(2)}).
		\qedhere
	\end{equation*}
\end{proof}

For the rest of this subsection we assume that $x_0=0$ and $\mu\equiv0$. Then, as in \S\ref{subsec:representation}, we can write $X_t=\tilde W_{[X]_t}$ where $\tilde W$ is a standard Brownian motion and $[X]$ is the quadratic variation of $X$. To proceed we need the following result about zooming in at the supremum of $\tilde W$, defined for the stochastic interval $[0,[X]_1]$. This result is essentially a direct consequence of \cite[Cor.~2]{iva_zooming} except for one technical complication. That paper works only on the canonical path space and since stable convergence is not only concerned with laws but also very much with the probability space the result does not apply directly. Instead we provide a short proof which fixes this problem.

\begin{lemma}\label{lem:zooming_BM}
	It holds that
	\begin{equation}\label{eq:zooming_BM}
		\left((\epsilon^{-1/2}\underleftarrow{\tilde W}_{\epsilon t})_{t\geq0},(\epsilon^{-1/2}\underrightarrow{\tilde W}_{\epsilon t})_{t\geq0}\right)\stably(-B^{(1)},-B^{(2)}),
	\end{equation}
	where $\underleftarrow{\tilde W}$ and $\underrightarrow{\tilde W}$ are the pre- and post-supremum processes defined for the interval $[0,[X]_1]$.
\end{lemma}
\begin{proof}
	For each $T>0$ we let $\underleftarrow{\tilde W}^{(T)}$ and $\underrightarrow{\tilde W}^{(T)}$ denote the pre- and post-supremum processes for $\tilde W$, defined for the interval $[0,T]$. According to \cite[Thm.~4]{iva_zooming} there is the stable convergence
	\begin{equation*}
		\left((\epsilon^{-1/2}\underleftarrow{\tilde W}^{(T)}_{\epsilon t})_{t\geq0},(\epsilon^{-1/2}\underrightarrow{\tilde W}^{(T)}_{\epsilon t})_{t\geq0}\right)\overset{\H-st}{\rightarrow}(-B^{(1)},-B^{(2)}),
	\end{equation*}
	where $\H$ is the $\sigma$-algebra generated by $\tilde W$. Since the left-hand side is obviously $\H$-measurable the $\H$-stable convergence extends to $\ff$-stable convergence by Lemma~\ref{lem:H-stable_to_F-stable}.

	At this point it remains to extend to the case $T=[X]_1$. Corollary~\ref{cor:max_not_at_1} tells us that the supremum of $\tilde W$ over the interval $[0,[X]_1]$ is almost surely attained strictly before time $[X]_1$. Using this the convergence in \eqref{eq:zooming_BM} follows via the same arguments as in the proof of \cite[Cor.~2]{iva_zooming}.
\end{proof}

Finally we are ready to prove Theorem \ref{thm:zoom_sup_diffusion} in the case with $x_0=0$ and $\mu\equiv0$. As in Lemma \ref{lem:zooming_BM} we let $\underleftarrow{\tilde W}$ and $\underrightarrow{\tilde W}$ denote the pre- and post-supremum processes for $\tilde W$ defined for the interval $[0,[X]_1]$.

Since $[X]_t=\int_0^t\sigma^2(X_s)\idd s$ it follows immediately that
\begin{equation*}
	\sigma^2_\epsilon(t):=\epsilon^{-1}([X]_{m^X+\epsilon t}-[X]_{m^X})\to t\sigma^2(\overline X)
\end{equation*}
a.s.\ for any $t\in\R$ since $\sigma$ is continuous on the range of $X$. We note that this convergence is uniform on compact sets. Hence we have the almost sure functional convergence
\begin{equation}\label{eq:timechange_diffusion}
	\left((\sigma^2_\epsilon(-t))_{t\geq0},(\sigma^2_\epsilon(t))_{t\geq0}\right)\to\left((-t\sigma^2(\overline X))_{t\geq0},(t\sigma^2(\overline X))_{t\geq0}\right),
\end{equation}
which we may add to the stable convergence in \eqref{eq:zooming_BM}. We further note that
\begin{equation*}
	\epsilon^{-1/2}\underrightarrow{X}_{\epsilon t}=\epsilon^{-1/2}(X_{m^X+\epsilon t}-\overline{X})=\epsilon^{-1/2}(\tilde W_{[X]_{m^X+\epsilon t}}-\tilde W_{[X]_{m^X}})=\epsilon^{-1/2}\underrightarrow{\tilde W}_{\epsilon\sigma^2_\epsilon(t)}
\end{equation*}
for each $t\geq0$. Similarly, it holds that $\epsilon^{-1/2}\underleftarrow{X}_{\epsilon t}=\epsilon^{-1/2}\underleftarrow{\tilde W}_{-\epsilon\sigma^2_\epsilon(-t)}$ for all $t\geq0$. By continuity of subordination (see \cite[Thm.~13.2.2]{whitt}) we have the convergence 
\begin{equation*}
	\left((\epsilon^{-1/2}\underleftarrow{X}_{\epsilon t})_{t\geq0},(\epsilon^{-1/2}\underrightarrow{X}_{\epsilon t})_{t\geq0}\right)\stably\left((-B^{(1)}_{t\sigma^2(\overline X)})_{t\geq0},(-B^{(2)}_{t\sigma^2(\overline X)})_{t\geq0}\right)=\left(-\sigma(\overline X)\tilde B^{(1)},-\sigma(\overline X)\tilde B^{(2)}\right),
\end{equation*}
where $\tilde B^{(i)}_t:=\sigma^{-1}(\overline X)B^{(i)}_{t\sigma^2(\overline X)}$. We note that $(\tilde B^{(1)},\tilde B^{(2)})$ is again a pair of Bessel-3 processes, independent of $\ff$ and of each other. This concludes the proof of Theorem \ref{thm:zoom_sup_diffusion}.

\section*{Acknowledgements}
I am thankful to my supervisor Jevgenijs Ivanovs for providing valuable feedback and for guiding me towards relevant literature.

Furthermore I gratefully acknowledge financial support of Sapere Aude Starting Grant 8049-00021B ``Distributional Robustness in Assessment of Extreme Risk'' from Independent Research Fund Denmark.

\printbibliography

\end{document}